\documentclass[a4paper,11pt,makeidx]{amsart}

\usepackage{amssymb,eucal,graphicx}

\usepackage{ifthen}
\usepackage{epsfig}
\usepackage{amsmath}
\usepackage{amsfonts}
\usepackage[T1]{fontenc}

\RequirePackage{amsthm}
\RequirePackage{amssymb}
\usepackage[all,ps]{xy}
\usepackage{latexsym}

\numberwithin{equation}{section}

\newtheorem{theorem}[equation]{Theorem}

\newtheorem{corollary}[equation]{Corollary}

\newtheorem{lemma}[equation]{Lemma}

\newtheorem{proposition}[equation]{Proposition}

\newtheorem{definition}[equation]{Definition}

\newtheorem{question}[equation]{Question}

\newtheorem{remark}[equation]{Remark}

\newtheorem{rmk}[equation]{Remark}

\def\rk{\operatorname{rk}}

\def\Sing{\operatorname{Sing}}

\def\log{\operatorname{log}}

\def\iso{\simeq}

\def\+{\oplus}                   
\def\*{\otimes}                  
\def\hpil{\longrightarrow}       

\def\Pic{\operatorname{Pic}}

\newcommand\Ii{{\mathcal I}}
\newcommand\Oc{{\mathcal O}}

\newcommand\KC{\mathcal KC}
\newcommand\B{\mathcal B}
\newcommand\M{\mathcal M}

\newcommand\Ha{\mathcal H}

\newcommand\N{\mathcal N}

\newcommand\T{\mathcal T}
\newcommand\V{\mathcal V}

\newcommand\ZZ{\mathbb{Z}}
\newcommand\CC{\mathbb{C}}

\newcommand\PP{\mathbb P}

\begin{document}

\title{Nodal curves with general moduli on $K3$ surfaces}

\author[F.~Flamini, A.~L.~Knutsen, G.~Pacienza, E.~Sernesi]{Flaminio Flamini$^{(1)}$, Andreas Leopold Knutsen$^{(2)}$, 
Gianluca Pacienza$^{(3)}$ and Edoardo Sernesi$^{(4)}$}


\thanks{(1) and (4): Member of MIUR-GNSAGA at INdAM "F. Severi"}

\thanks{(2): Research supported by a Marie Curie Intra-European Fellowship within the 6th Framework Programme}

\thanks{(3): During the preparation of the paper the author benefitted from an "accueil en d\'el\'egation au CNRS"}

%

\vskip -30 pt

\begin{abstract} {We investigate the modular properties of nodal curves on a low genus $K3$ surface. We prove that a general genus $g$ curve 
$C$ is the normalization of a $\delta$-nodal curve $X$ sitting 
on a primitively polarized $K3$ surface $S$ of degree $2p-2$, for $2 \le g=p-\delta< p \le 11$. 
The proof is based 
on a local deformation-theoretic analysis of the 
map from the stack of pairs  $(S,X)$ to the moduli space of curves $\M_g$ that associates to $X$ the isomorphism class $[C]$ of its  normalization.}

\end{abstract}

\maketitle

%
\section{Introduction}\label{S:intro}
%
 Nonsingular curves of low genus on a K3 surface have  interesting modular properties,  
  related to the existence of Fano 3-folds of index one of the corresponding sectional genus. These properties have been investigated by  Mukai  who settled, in particular, a  problem raised by Mayer in \cite{Ma}.  He showed that  a general curve of genus $g \le 9$ or $g=11$ can be embedded as a nonsingular curve in a K3 surface, and that this is not possible for curves of genus $g=10$, despite an obvious count of constants indicating  the opposite.   These facts have been proved again by Beauville in the last section of \cite{B} from a different point of view, by means of a local deformation-theoretic analysis. 
 
 In the  present paper we take a  point of view similar to Beauville's with the purpose of studying the corresponding questions about moduli of \emph{singular (nodal) curves} of low genus on a K3 surface. 
 To this end we
 consider the  following algebraic stacks: 
\vspace{0,4cm} 

 \begin{itemize}
 \item[${\B}_p$:]   the stack     of  smooth $K3$ surfaces $S$ marked by a globally generated, primitive line bundle $H$ of sectional genus $p \ge 2$; it is smooth and irreducible of dimension $19$.  
 \item[]
 \item[$\V_{p,\delta}$:] the stack of pairs  $(S,X)$  such that
$(S,H) \in {\B}_p$ and $X \in |H|$ is an irreducible curve with $\delta$ nodes and no other singularities,  for given $0 \le \delta \le p$;  it is smooth of dimension 
  $19+g$, where $g=p - \delta$.
\end{itemize}
\vspace{0,4cm}  

  We also consider an {\'e}tale atlas $V_{p,\delta} \to \V_{p,\delta}$ and the   morphisms:
 \[
\xymatrix{
 V_{p,\delta}\ar[d]^{\pi_\delta} \ar[r]^{c_{p,\delta}} &   \M_g \\
 {\B}_p}
\]
where $\M_g$ is the moduli stack of nonsingular curves of genus $g =p-\delta$;  $c_{p,\delta}$ and $\pi_\delta$ are  induced by associating to a point parametrizing a pair $(S,X)$ the isomorphism class of the normalization of $X$ and $[S]$ respectively.

  We study this configuration when $3 \le p \le 11$. Our main result is the following:
 
 \begin{theorem}
 Let  $3 \le p \le 11$ and $0 \le \delta \le p-2$, so that  $2 \le g =p-\delta \le p$.  Let 
 $V \subset  V_{p,\delta}$ be an irreducible component, and let 
 \[
 c_{p,\delta|V}: V \longrightarrow \M_g
 \] be the restriction to $V$ of the morphism $c_{p,\delta}$.  Then, for any $2 \le g < p \le 11$,  
 $c_{p,\delta|V}$ is dominant.  In particular, for the general fibre $F_{p,\delta}$ of $c_{p,\delta |V}$ we have:
 \[
 \dim(F_{p,\delta}) = 22 - 2g = {\rm expdim}(F_{p,\delta})
 \]

 \end{theorem}

 The theorem is proved  by  studying deformations of a pair
 $(S,X)$  in  $\V_{p,\delta}$. The locally trivial deformation  theory of such a pair  is controlled by a locally free sheaf of rank 2, namely the sheaf $\T_S\langle X\rangle$ of tangent vectors to $S$ that are tangent to $X$.  Specifically, $H^1(S, \T_S\langle X\rangle)$ is the tangent space to  $\V_{p,\delta}$ at $(S,X)$ and 
 $H^2(S, \T_S\langle X\rangle)$ is an obstruction space. These cohomology groups are studied by pulling back $\T_S\langle X\rangle$ to the blow-up $\tilde S$ of $S$ at the singular points of $X$. 
 Then the local study of the morphism $c_{p,\delta}$ is carried out on $\tilde S$, and the theorem is   reduced to proving the vanishing of an appropriate cohomology group.  
 
 The case $p=11,\ \delta=1$ of the theorem appears to be somehow unexpected.  Note that, in fact, the theorem says that every irreducible component $V$ of $V_{11,1}$  dominates $\M_{10}$. This contrasts the fact that, according to Mukai,   $V_{10,0}$ does not dominate $\M_{10}$.   
 
 Another interesting case is $p=10, \delta=1$.  Again the theorem says that every irreducible component of $V_{10,1}$ dominates $\M_9$.  But, since $V_{10,0}$ is mapped to a divisor of $\M_{10}$ by $c_{10}:=c_{10,0}$, it follows that the nodal curves in $V_{10,1}$ only fill a divisor on the boundary 
 $\partial\overline{\M}_{10}$ of  $\overline{\M}_{10}$, despite the fact that their normalizations  are  general curves of genus $9$.  This means that on a general curve $C$ of genus $9$, the effective divisors $P+Q$, with $P \ne Q$, such that the nodal curve $X=C/(P=Q)$ can be embedded in a K3 surface,  belong to a 1-dimensional cycle $\Gamma \subset C^{(2)}$.  It would be interesting to compute the numerical class of $\Gamma$.   
 
 The paper consists of  5 sections including the introduction. After recalling the relevant deformation theory in 
 \S\;2, we survey the known results about moduli of smooth curves on marked K3 surfaces in \S\;3. In \S\;4 we develop our approach for the case of nodal curves, and in \S\;5 we discuss the existence of nodal curves  having normalizations with general moduli. In the end we raise some related open questions.

\vspace{0,5cm}

{\it Acknowledgements.} We warmly thank B.~Fantechi and M.~Roth for useful conversations and C. Voisin for her comments.

%
\section{Some basic results of deformation theory}\label{S:deftheory}
%

In this section we will review some results on deformation theory
that are needed for our aims. For complete details, we refer the
reader to e.g. \cite[\S\;3.4.4]{Ser}.

Let $Y$ be a smooth variety and let $j: X \hookrightarrow Y$ be a closed embedding of a Cartier divisor $X$.
The  {\em locally trivial deformations} of $j$ are studied by means of suitable sheaves on $Y$.

Let $\N_{X/Y}$ be the normal 
sheaf of $X$ in $Y$, and $\N'_{X/Y} \subseteq \N_{X/Y}$ the {\em equisingular normal sheaf} of $X$ in $Y$
(cf.  \cite[Proposition 1.1.9]{Ser}).
One can define  a coherent sheaf $\T_Y \langle X\rangle$ of rank $\dim(Y)$ on 
$Y$  via the exact sequence :
\begin{equation}\label{eq:txy}
0 \longrightarrow \T_Y \langle X \rangle \longrightarrow \T_Y {\longrightarrow} \N'_{X/Y} \longrightarrow 0,
\end{equation}which is called the  {\em sheaf of germs of tangent vectors to $Y$
that are tangent to $X$} (cf. \cite[\S\;3.4.4]{Ser}). Of course, when $X$ is smooth, then $\N'_{X/Y}$ in \eqref{eq:txy} is nothing but the normal bundle $\N_{X/Y}$.

One has  a natural surjective restriction map
\begin{equation}\label{eq:r}
r : \T_Y \langle X \rangle  \longrightarrow \T_X,
\end{equation}giving the exact sequence
\begin{equation}\label{eq:exseqr}
0 \longrightarrow  \T_Y(-X) \longrightarrow   \T_Y \langle X \rangle \longrightarrow \T_X \longrightarrow 0,
\end{equation}where $\T_Y(-X)$ is the vector bundle of tangent vectors of $Y$
vanishing along $X$ and where $\T_X$ is the {\em tangent sheaf} of
$X$, i.e. the dual sheaf of the sheaf of K\"ahler differentials of
$X$  (cf. \cite[\S\;3.4.4]{Ser}).

Observe that, when $X$ is a divisor  with simple normal crossings (see \cite{Ii}), 
  $\T_Y \langle X \rangle$ 
is a locally free subsheaf of the holomorphic tangent bundle $\T_{Y},$ whose restriction to $X$ is $\T_{X}$ and whose
localization at any point $x\in X$ is given by
\begin{equation*}
{\T}_{Y,x}=\underset{i=1}{\overset{l}{\sum }}\mathcal{O}_{Y,x}\cdot z_{i}\frac{\partial }{\partial z_{i}}+\underset{j=1+1}{\overset{n}{\sum }}\mathcal{O}_{Y,x}\cdot \frac{\partial }{\partial z_{j}},
\end{equation*}
where the local coordinates $z_1,z_2,\ldots ,z_{n}$ around $x$ are
chosen in such a way that $X=\{ z_{1}z_2 \cdots z_{l}=0\}$. In fact, $\T_Y \langle X \rangle = 
(\Omega_Y^1(\log X))^{\vee}$, where 
$\Omega_Y^1(\log X)$ denotes the sheaf of meromorphic $1$-forms on $Y$ that 
have at most logarithmic poles along $X$.  

Also, when $X$ is an integral curve sitting on a smooth surface $Y$, by (\ref{eq:txy}) the sheaf 
$\T_Y \langle X \rangle$ is an elementary transformation of the locally free sheaf $\T_Y$, and then it is locally free (see e.g. \cite[Lemma 2.2]{gom}).

Recall the following basic result:

\begin{proposition}\label{prop:beau} (see \cite[Proposition 3.4.17]{Ser})
The locally trivial deformations of the pair $(Y,X)$
(equivalently of the closed embedding $j$) are controlled by the
sheaf $\T_Y \langle X \rangle$; namely,
\begin{itemize}
\item the obstructions lie in $H^2(Y,\T_Y \langle X \rangle)$;
\item first-order, locally trivial deformations are parametrized by $H^1(Y,\T_Y \langle X \rangle)$;
\item infinitesimal automorphisms are parametrized by $H^0(Y,\T_Y \langle X \rangle)$.
\end{itemize}

The map that associates to a first-order, locally trivial
deformation of $(Y,X)$ the corresponding first-order deformation of $X$ is the
map
\begin{equation}\label{eq:h1r}
H^1(r) : H^1(Y,\T_Y \langle X \rangle) \longrightarrow H^1(X,\T_X),
\end{equation}induced in cohomology by \eqref{eq:r}.
\end{proposition}

In the rest of the paper we will focus on the case
of nodal curves on a surface. 

%
\section{Mukai's results on smooth, canonical curves on general, marked $K3$ surfaces and
Beauville's infinitesimal approach}\label{S:MukBea}
%

In this section, we shall briefly recall some results of Mukai \cite{M1,M2}  
and the infinitesimal approach considered by Beauville \cite[\S\;5]{B}.

Let $p \geq 2$ be an integer. 
Let $\B_p$ be the moduli stack of  smooth $K3$ surfaces marked by a globally generated, primitive line bundle of sectional genus $p$. That is, the elements of $\B_p$ are pairs $(S,H)$ where $S$ is a smooth $K3$ surface and $H$ is a globally generated line bundle on $S$ with $H^2=2p-2$ and such that $H$ is nondivisible in $\Pic (S)$.  It is well-known that $\B_p$ is smooth, irreducible and of dimension $19$ (cf. e.g.
\cite[Thm.VIII 7.3 and p.~366]{BPV} for the scheme structure; the same conclusions hold 
also for the stack structure of $\B_p$).

\begin{definition}\label{def:stack} Let $\KC_p$ be the algebraic stack of pairs $(S,C)$,  
where $(S,H) \in \B_p$, $p \geq 2$, and $C \in |H|$ is a smooth irreducible 
curve.   
\end{definition}

\noindent
Observe that there is an induced, surjective morphism of stacks
\begin{equation} \label{eq:pi}
\pi : \KC_p \hpil \B_p
\end{equation}given by the natural projection. 
From  \cite[\S\;(5.2)]{B}, for any $(S,C)
\in \KC_p$,  by Serre duality one has

\begin{equation}\label{eq:bea1}
H^2(S, \T_S \langle C
\rangle) = H^0(S, \Omega_S^1(\log C))^{\vee} = (0).
\end{equation}

Furthermore, since $C$ is a smooth curve of genus 
$p \geq 2$ and since $\T_S \cong \Omega_S^1$, being $S$ a $K3$ surface and $\T_S$ a rank-two vector bundle 
on it,   by \eqref{eq:exseqr} we have $$H^0(S, \T_S \langle C \rangle) = (0).$$In particular, 
from Proposition \ref{prop:beau}, $\KC_p$ is a smooth stack  of dimension 
$${\rm dim} (\KC_p) = h^1(S, \T_S \langle C \rangle) = 19 + p.$$ 
Since the fibers of $\pi$ are connected,  
$\KC_p$ is also irreducible.

Let $\M_p$ be the moduli stack of smooth curves of genus $p$, which is irreducible and of dimension
$3p-3$, since $p \geq 2$ by assumption. One has a natural morphism of stacks
\begin{equation}\label{eq:cp}
c_p : \KC_p \hpil \M_p
\end{equation}defined as $$c_p((S,C)) = [C] \in \M_p,$$where $[C]$ denotes
the isomorphism class of $C \subset S$.

Observe that:
\begin{eqnarray}\label{eq:dimp}
\begin{array}{ccl}
 {\rm dim}(\KC_p) > {\rm dim} (\M_p), & \; {\rm for} \; & p \leq 10, \\
 {\rm dim}(\KC_p) = {\rm dim} (\M_p), &  \; {\rm for} \; & p = 11, \\
 {\rm dim}(\KC_p) < {\rm dim} (\M_p), & \; {\rm for} \; & p \geq 12.
\end{array}
\end{eqnarray}

If we denote by $F_p$ the general fibre of $c_p$, then the {\em expected dimension} 
of $F_p$ is:

\begin{equation}\label{eq:expdim}
{\rm expdim} (F_p) = \Big\{ \begin{array}{ccl}
22 - 2 p & \; {\rm for} \; & p \leq 10, \\
 0 & \; {\rm for} \; & p \geq 11.
\end{array}
\end{equation}

The main results concerning the morphism $c_p$ are contained in the following:

\begin{theorem}[Mukai]\label{thm:mukai} With notation as above:
\begin{itemize}
\item[(i)] $c_p$ is dominant for $p \leq 9$ and $p = 11$ (cf. \cite{M1});
\item[(ii)] $c_p$ is not dominant for $p = 10$ (cf. \cite{M1}). More precisely, its image is a
hypersurface in $\M_{10}$ (cf. \cite{CU});
\item[(iii)] $c_p$ is generically finite onto its image, for $p = 11$ and for $p \geq 13$,
but not for $p = 12$ (cf. \cite{M2}).
\end{itemize}
\end{theorem}

\vspace{0,4cm}

\begin{remark}\label{rem:expdim} {\normalfont
(1) In particular, from \eqref{eq:expdim} and from Theorem \ref{thm:mukai}, one has
$${\rm dim} (F_p) = {\rm expdim} (F_p),$$
unless either 
\begin{itemize}
\item $p=10$, in which case ${\rm dim} (F_p) = {\rm expdim} (F_p) + 1 = 3$, or
\item $p=12$, in which case $ {\rm dim} (F_p) \geq 1$. 
\end{itemize}

\noindent
(2) When the map $c_p$ is not dominant, one can look at it as a way to produce 
hopefully interesting cycles in the moduli space of curves. 
The case $p=10$ is particularly relevant, as the 
divisor in $\M_{10}$ parametrizing curves lying on a $K3$ surface was the first counterexample to the 
slope conjecture (see \cite{FP}).}
\end{remark}

\vspace{0,4cm}

In \cite[\S\;(5.2)]{B} Beauville considered the morphism $c_p$
from a differential point of view. Let $(S,C) \in \KC_p$ be any point. From Proposition \ref{prop:beau}, 
the differential of $c_p$ at the point $(S,C)$ can be identified with the map
$$H^1(r) : H^1(\T_S \langle C \rangle) \longrightarrow H^1(\T_C),$$as 
in \eqref{eq:h1r}. From \eqref{eq:exseqr} and \eqref{eq:bea1} it follows that $H^1(r)$ fits in the exact sequence:
$$0 \longrightarrow H^1(S,\T_S(-C)) \longrightarrow H^1(S,\T_S \langle C \rangle) \stackrel{H^1(r)}{\longrightarrow} 
H^1(C,\T_C) \longrightarrow H^2(S,\T_S(-C)) \longrightarrow 0.$$

Using Serre duality and the fact that $\omega_S$ is trivial, we get
\begin{equation}\label{eq:flam}
H^j(S,\T_S(-C)) \cong H^{2-j} (S, \Omega^1_S (C))^{\vee}, \;\;\; 0 \leq j \leq 2.
\end{equation}

From \eqref{eq:flam} we obtain that the morphism $c_p$ is:
\begin{itemize}
\item {\em smooth} at $(S, C) \in \KC_p$ (i.e. the differential $(c_{p})_*$ at the
point $(S,C)$ is surjective) if and only if $H^0(S,\Omega^1_S(C)) = (0)$;
\item {\em unramified} at $(S, C)\in \KC_p$ (i.e. the differential $(c_{p})_*$ at the
point $(S,C)$ is injective) if and only if $H^1(S, \Omega_S^1(C)) = (0)$.
\end{itemize}Note that the above conditions depend only on the 
 marking  
$H = \Oc_S(C)$ and not on the particular curve $C$ in $|H|$.

With this approach, Theorem \ref{thm:mukai} is equivalent to:

\begin{proposition}\label{prop:beaumuk}{\em(see \cite[\S\;(5.2)]{B}).}
Let $(S, H)$ be a general  primitively polarized $K3$ surface of 
genus $p\geq 2$.   We have:
\begin{itemize}
\item[(i)] $h^0(S,\Omega_S^1 (H)) = 0$, for $p \leq 9$ and $p = 11$;
\item[(ii)] $h^0(S,\Omega_S^1 (H)) = 1$, for $p = 10$;
\item[(iii)] $h^1(S,\Omega_S^1 ( H) ) = 0$, for $p = 11$ and $p \geq 13$;
\item[(iv)] $h^1(S,\Omega_S^1 ( H)) \geq 1$, for $p = 12$.
\end{itemize}
\end{proposition}

\begin{rmk} \label{rem:accazero}
{\rm Since $h^2(\Omega_S^1 (H)) = h^0(\T_S(-H))=0$, $c_1(\Omega_S^1 (H))=2H$ and
$c_2(\Omega_S^1 (H))=H^2+24=2p+22$, we have, by Riemann-Roch
\begin{eqnarray*}
h^0(\Omega_S^1 (H)) &  =   & \chi(\Omega_S^1 (H)) + h^1(\Omega_S^1 (H)) \\
                    & = & \frac{c_1(\Omega_S^1 (H))^2}{2} -  c_2(\Omega_S^1 (H)) + 2 \rk (\Omega_S^1 (H)) + h^1(\Omega_S^1 (H))\\
&  =   & 2p-22 + h^1(\Omega_S^1 (H)) \geq 2p-22.
\end{eqnarray*}
In particular, $h^0(\Omega_S^1 (H)) \geq 3$ if $p \geq 12$ (cf. also Question \ref{quest:pd}).}
\end{rmk}

%
\section{The approach to the nodal case}\label{S:Tessier}
%

By using Proposition \ref{prop:beau} and a similar approach as in
\S\;\ref{S:MukBea}, we want to deduce some extensions of Theorem
\ref{thm:mukai} to irreducible, nodal curves in the primitive
linear system $|H|$ on a general  primitively polarized $K3$ surface of 
genus $p \geq 3$.  In particular, we are interested in determining when the normalization of such a singular curve is 
an (abstract) smooth curve with general moduli. 

To do this, we have to fix some notation and to prove some 
results that will be used in what follows. First we recall that, for any smooth surface $S$ and any line
bundle $H$ on $S$, such that $|H|$ contains smooth, irreducible
curves of genus $p:= p_a(H)$, and any positive integer $\delta \leq p$, one
denotes by$$V_{|H|, \delta}(S) \; \; \mbox{or simply} \; V_{|H|, \delta}$$the locally closed and
functorially defined subscheme of $|H|$ parametrizing the
universal family of irreducible curves in $|H|$ having
$\delta$ nodes as the only singularities and, consequently, geometric genus $g:= p- \delta$. These are classically called {\em Severi varieties} of irreducible, $\delta$-nodal curves on $S$ in $|H|$.

It is well-known, as a direct consequence of Mumford's
theorem on the existence of nodal rational curves on $K3$ surfaces
(see e.g. \cite[pp.~365-367]{BPV}) and standard results on Severi
varieties (see e.g. \cite{Tan,CS,F}),  that if  $(S,H) \in \B_p$ is 
general, $p \geq 2$,  
then 
\begin{equation}\label{eq:reg}
V_{|H|, \delta} \textrm { is nonempty and 
regular,} 
\end{equation}
i.e. it is smooth and (each of its irreducible components is) 
of the expected dimension $ g = p-\delta$, for each $\delta \leq p$. (In fact, the regularity holds whenever
$V_{|H|, \delta}$ is nonempty.)

From now on, we shall always consider 
\begin{equation}\label{eq:boundpdelta}
p \geq 3\;\; {\rm and} \;\;\; 0 \leq \delta \leq p-2, \;\; {\rm so \; that} \;\; g \geq 2. 
\end{equation}
Similarly as in Definition \ref{def:stack}, we have: 
\begin{definition}\label{def:stackb} For any  $p$ and $\delta$ as in \eqref{eq:boundpdelta}, let 
$\mathcal{V}_{p,\delta} $ be the stack of pairs $(S,X)$, such that 
 $(S,H) \in \B_p$ and $[X] \in V_{|H|,\delta}(S)$.  
\end{definition}Of course $\mathcal{V}_{p,0} = \KC_p$ as in Definition \ref{def:stack}.  

For any fixed $p$ and any $\delta$ as in \eqref{eq:boundpdelta}, 
the stacks $\mathcal{V}_{p,\delta} $ are locally closed substacks of a natural 
enlargement $\overline{\KC}_p$ of $\KC_p$, which is defined as the stack of pairs $(S,C)$, where  $(S,H) \in \B_p$ and $C \in |H|$. 
It follows that the stacks $\mathcal{V}_{p,\delta} $ are algebraic because  $\overline{\KC}_p$ is.

Consider $\B_p^0 \subset \B_p$ the open dense substack parametrizing elements $(S,H)$ in $\B_p$ that verify 
\eqref{eq:reg}.

For any $\delta$ as above, let
\begin{equation} \label{eq:pib}
\pi_{\delta}: \mathcal{V}_{p,\delta} \hpil \B_p^0
\end{equation}
be the surjective morphism given by the projection.

Let $(S,X) \in \mathcal{V}_{p,\delta}$. From \eqref{eq:txy}, we can consider the exact sequence
\begin{equation}\label{eq:tsx}
0 \longrightarrow  \T_S \langle X\rangle \longrightarrow   \T_S \longrightarrow \N'_{X/S} \longrightarrow 0.
\end{equation}Since $S$ is a $K3$ surface  
$$H^2(\T_S) \cong H^{0}(\Omega^1_S) \cong H^{0}(\T_S)= (0).
$$
Therefore, passing to cohomology in \eqref{eq:tsx}, we get
\begin{equation}\label{eq:tsx1}
0 \longrightarrow H^0(\N'_{X/S}) \longrightarrow H^1(\T_S \langle X\rangle) \stackrel{H^1(r)}{\longrightarrow} H^1(\T_S)
\stackrel{\alpha}{\longrightarrow} H^1(\N'_{X/S}) \longrightarrow H^2(\T_S \langle
X\rangle) \longrightarrow 0,
\end{equation}where $H^1(r)$ is as in \eqref{eq:h1r}. 

From \cite[Lemma 2.4 and Theorem 2.6]{Tan}, we have
\begin{equation}\label{eq:tan}
H^1(\N'_{X/S}) \cong \CC, \; H^0(\N'_{X/S}) \cong T_{[X]}(V_{|H|, \delta}(S)).
\end{equation}Moreover, recall that $V_{|H|, \delta}(S)$ is regular at $[X]$ (i.e. it is 
smooth at $[X]$ and of the expected dimension $g$). This means that, despite the fact that
$h^1(\N'_{X/S}) =1$, the infinitesimal, locally trivial
deformations of the closed embedding $X
\stackrel{j}{\hookrightarrow} S$, with $S$ fixed, are unobstructed 
and the nodes impose independent conditions to
locally trivial deformations (cf. \cite[Remark 2.7]{Tan}).

We want to show that, for $\delta >0$, the properties of the stack $\mathcal{V}_{p,\delta}$ are similar 
to those of the stack $\KC_p$. 

\begin{proposition}\label{prop:smst} Let 
$p$, $\delta$ and $g$ be positive integers as in \eqref{eq:boundpdelta}. 
 
Then, for any $(S,X) \in \mathcal{V}_{p,\delta}$, we have
\begin{equation}\label{eq:palle}
h^0(\T_S \langle X \rangle) = 0 , \;\;\; h^2(\T_S \langle X \rangle) = 0.
\end{equation}

In particular, 
\begin{itemize}
\item[(i)] the stack $\mathcal{V}_{p,\delta}$ is smooth. 
\item[(ii)] Any irreducible component $\mathcal{V} \subseteq \mathcal{V}_{p,\delta}$  has 
dimension $h^1(\T_S \langle X \rangle) =19+g$.
\item[(iii)] The morphism $\pi_{\delta}$ is smooth and any irreducible component $\mathcal{V} \subseteq \mathcal{V}_{p,\delta}$ smoothly dominates $\B_p^0$. 
\end{itemize}
\end{proposition}

\begin{proof} The first equality in \eqref{eq:palle} directly follows from \eqref{eq:tsx}.

For what concerns the second equality in \eqref{eq:palle}, we consider \eqref{eq:tsx1} above: 
for any 
point $(S,X) \in \mathcal{V}_{p,\delta}$ we have that
\begin{equation}\label{eq:tsx2}
0 \longrightarrow H^0(\N'_{X/S}) \longrightarrow H^1(\T_S \langle
X\rangle) \longrightarrow {\rm Ker} (\alpha) \longrightarrow 0
\end{equation}
can be read as the natural differential sequence
\begin{equation}\label{eq:tsx3}
0 \longrightarrow T_{[X]} (V_{|H|, \delta} (S)) \longrightarrow  T_{(S,X)}
(\mathcal{V}_{p,\delta})  \longrightarrow T_{[S]} (\B^0_p)\longrightarrow 0;
\end{equation}
indeed, $\B^0_p$ is smooth of dimension $19$, whereas $h^1(\T_S) = 20$ and $h^1(\N'_{X/S}) = 1$
by \eqref{eq:tan}, thus we have 
${\rm Im} (\alpha) = H^1(\N'_{X/S}) \cong H^1(\N_{X/S})$, i.e. the elements of ${\rm Ker} (\alpha)$ 
can be identified with the first-order deformations of $S$ preserving the genus $p$  marking.

Moreover, it follows that $H^2(\T_S \langle X\rangle) = (0)$,
i.e. the infinitesimal deformations of the closed embedding $X
\stackrel{j}{\hookrightarrow} S$, with $S$ not fixed,  are
unobstructed. 

Now, (i) directly follows from Proposition \ref{prop:beau}, whereas (ii) and (iii) follow from Proposition 
\ref{prop:beau}, \eqref{eq:tsx3} and from what recalled above on Severi varieties on general $K3$ surfaces.
\end{proof}

In particular, we have:

\begin{corollary}\label{cor:smst} There exists an open, dense substack ${\mathcal U}_p \subseteq \B_p^0$ such that 
the number of irreducible components of $V_{|H|,\delta}(S)$ is constant, for $(S,H)$ varying in ${\mathcal U}_p$. 
\end{corollary}

\begin{proof} This directly follows from Proposition \ref{prop:smst} (iii) and from the fact that,  
$\mathcal{V}$ being irreducible, the general fibre of $ \pi_{\delta}|_{\mathcal{V}}$ has to be irreducible. 
\end{proof}

\begin{rmk}{\em Notice that it is not known
whether, for a general primitively polarized K3 surface $(S,H)$, the Severi varieties $V_{|k H|,\delta}(S)$, 
for any integer $k \geq 1$, are irreducible as soon as they have positive dimension. Interesting work on this topic has been done recently by Dedieu in \cite{D}.}
\end{rmk}

We need a reformulation of Proposition \ref{prop:beau} in this specific situation, that will be particularly useful for our aims.

First we have to recall some general facts. Let $ (S,X) \in \mathcal{V}_{p,\delta}$ be any pair as above.
Let$$\varphi : C \longrightarrow X \subset S$$ be the normalization morphism. By the fact that $X$ is a curve on $S$, 
one can consider the exact sequence:
\begin{equation}\label{eq:nfi}
0 \longrightarrow \T_C \longrightarrow \varphi^*(\T_S) \longrightarrow \N_{\varphi} \longrightarrow 0,
\end{equation}defining $\N_{\varphi}$ as the {\em normal sheaf} to  $\varphi$.
Since $X$ is a nodal curve, it is well-known that $\N_{\varphi}$ is a line bundle on $C$ 
(cf. e.g. \cite{Ser}).

\begin{remark}\label{rem:spezzamento}
{\em Denote as above by $\varphi : C \to X \subset S$
the normalization morphism of a curve sitting on a surface $S$. Suppose moreover that $S$ is a $K3$
surface. Then $\N_{\varphi}=\omega_C$. 
Therefore, unless $C$ is hyperelliptic, we always have the surjectivity of the map :
$$
 H^0(C, \N_{\varphi})\otimes H^0 (C, \omega_C^{\otimes 2})\longrightarrow H^0(C, \N_{\varphi}\otimes \omega_C^{\otimes 2}). 
$$
Arguing as in \cite[Proposition 1]{BM}, we get that 
the splitting of the exact sequence (\ref{eq:nfi}) is {\it equivalent} to the triviality, as abstract deformations of $C$, of the infinitesimal deformations parametrized by $H^0(C,\N_{\varphi})$, i.e. the coboundary map 
$H^0(C,\N_{\varphi})\longrightarrow H^1(C,\T_C)$ is zero.}
\end{remark}

\begin{lemma}\label{lem:HoriTan} 
With notation as above, one has
$$
\varphi_*( \N_{\varphi}) \cong \N'_{X/S}.
$$
In particular, 
\begin{equation}\label{eq:tan2}
H^i (\N'_{X/S}) \cong H^i(\N_{\varphi}), \;\; 0 \leq i \leq 1
\end{equation}
\end{lemma}

\begin{proof} The reader is referred to \cite[p. 111]{Tan2} where, with the notation therein, 
since $X$ is nodal, 
the {\em Jacobian ideal} ${\mathfrak J}$ coincides with the {\em conductor ideal} ${\mathfrak C}$, i.e. 
$\N'_{X/S} = {\mathfrak J}  \; \N_{X/S} = {\mathfrak C} \; \N_{X/S} \cong \varphi_*( \N_{\varphi})$. 
\end{proof}  

On the other hand, if $N = \Sing(X)$, 
let $$\mu_N : \widetilde{S} \longrightarrow S$$be the blow-up of $S$ along $N$. Thus, 
$\mu_N$ induces the embedded resolution of $X$ in $S$, i.e. we have the following commutative diagram:
\begin{displaymath}
\begin{array}{clccl}
C &  & \subset & \tilde{S} & \\
\downarrow & \!\!\!\!^{\varphi} & & \downarrow & \!\!\!^{\mu_N} \\
X & & \subset & S & ,
\end{array}
\end{displaymath}
where  $\mu_N|_C = \varphi$.

Since $$\varphi^*(\T_S) \cong \mu_N^*(\T_S) \otimes \Oc_C$$and since 
$\mu_N|_C = \varphi$, we also have the exact sequence on $\tilde{S}$:
\begin{equation}\label{eq:stilde1}
0 \longrightarrow \mu_N^*(\T_S) (-C) \longrightarrow \mu_N^*(\T_S) \longrightarrow \varphi^*(\T_S) \longrightarrow 0.
\end{equation}We denote by 
\begin{equation}\label{eq:lambda}
\mu_N^*(\T_S) \stackrel{\lambda}{\longrightarrow}  \N_{\varphi}\longrightarrow 0
\end{equation}the composition of the two surjections in  \eqref{eq:stilde1} and \eqref{eq:nfi}.

\begin{definition}\label{def:kerlambda} With notation as above, let
$$ \mu_N^*(\T_S) \langle C\rangle := {\rm Ker} (\lambda).$$
\end{definition}From \eqref{eq:nfi}, \eqref{eq:stilde1} and \eqref{eq:lambda}, 
the sheaf $\mu_N^*(\T_S)\langle C\rangle$ sits in the following natural exact diagram:
 \begin{equation}\label{eq:diag1}
\begin{array}{lcccccr}
       & 0                              &      & 0                  &     &                & \\
       & \downarrow                     &      & \downarrow         &     &                & \\
0 \longrightarrow  & \mu_N^*(\T_S) (- C)            & \stackrel{=}{\longrightarrow}  & \mu_N^*(\T_S) (-C) & \longrightarrow &        0       & \\
       &  \downarrow                    &      & \downarrow         &     &           \downarrow     & \\
0 \longrightarrow  & \mu_N^*(\T_S) \langle C\rangle & \longrightarrow  & \mu_N^*(\T_S)      & \stackrel{\lambda}{\longrightarrow} &   \N_{\varphi} & \longrightarrow 0\\
       & \downarrow^{\tau}                     &      & \downarrow         &     &     ||         & \\
0 \longrightarrow  & \T_C                           & \longrightarrow  & \varphi^*(\T_S)    & \longrightarrow &   \N_{\varphi} & \longrightarrow 0\\
       & \downarrow                     &      & \downarrow         &     &     \downarrow & \\
       & 0                              &      & 0                  &     &         0      &
\end{array}
\end{equation}

We have

\begin{proposition}\label{prop:flam} Let $(S,X) \in \mathcal{V}_{p,\delta}$. Then 
\begin{equation}\label{eq:iso}
H^i(S, \T_S \langle X \rangle) \cong H^i (\widetilde{S}, \mu_N^*(\T_S) \langle C \rangle), \; \; 0 \leq i \leq 2.
\end{equation}

In particular, the locally trivial deformations of the pair $(S,X)$ are also governed by the sheaf 
$ \mu_N^*(\T_S) \langle C \rangle$, i.e. 
\begin{itemize}
\item the obstructions lie in $H^2(\widetilde{S}, \mu_N^*(\T_S) \langle C \rangle)$;
\item first-order, locally trivial deformations are parametrized by $H^1(\widetilde{S}, \mu_N^*(\T_S) \langle C \rangle)$;
\item infinitesimal automorphisms are parametrized by $H^0(\widetilde{S}, \mu_N^*(\T_S) \langle C \rangle)$.
\end{itemize}

\end{proposition}
\begin{proof} Observe that, in this situation, we have the exact sequence \eqref{eq:tsx}, i.e.:
\begin{equation}\label{eq:ganza1}
0 \longrightarrow \T_S \langle X \rangle \longrightarrow \T_S \longrightarrow \N'_{X/S} \longrightarrow 0. 
\end{equation}On the other hand, as above, let
$\mu_N : \; \widetilde{S} \longrightarrow S $ be the blow-up of $S$ along $N$ and let 
$C \subset \widetilde{S}$ be the proper transform of $X \subset S$. From the second row of \eqref{eq:diag1}, we get
\begin{equation}\label{eq:ganza2}
0 \longrightarrow \mu_N^*(\T_S) \langle C \rangle \longrightarrow \mu_N^*(\T_S) \longrightarrow \N_{\varphi} \longrightarrow 0. 
\end{equation}Since $S$ is smooth then, by the projection formula, we have 
$$
{\mu_N}_*(\mu_N^*(\T_S))  \cong  \T_S \otimes {\mu_N}_*(\Oc_{\widetilde{S}}) \cong \T_S.
$$
Now we apply ${\mu_N}_*$ to \eqref{eq:ganza2}, getting
\begin{equation}\label{eq:ganza3}
0 \longrightarrow {\mu_N}_*(\mu_N^*(\T_S) \langle C \rangle) \longrightarrow {\mu_N}_*(\mu_N^*(\T_S)) \longrightarrow \varphi_*(\N_{\varphi}),  
\end{equation}where the equality 
${\mu_N}_*(\N_{\varphi}) = \varphi_*(\N_{\varphi})$ directly follows from the facts that $\N_{\varphi}$ 
is a line-bundle on $C$ and that, as above, ${\mu_N}|_C = \varphi$.

By using the exact sequences \eqref{eq:ganza1} and \eqref{eq:ganza3}, together with 
Lemma \ref{lem:HoriTan}, we get: 

\begin{equation}\label{eq:diag2}
\begin{array}{lcccccr}
0 \longrightarrow  & \T_S \langle X \rangle & \longrightarrow  & \T_S   & \longrightarrow &   \N'_{X/S}  & \longrightarrow 0\\
       & \downarrow                     &      & \downarrow^{\cong}         &     &        \downarrow^{\cong}     & \\
0 \longrightarrow  & {\mu_N}_*(\mu_N^*(\T_S) \langle C \rangle )         & \longrightarrow  & {\mu_N}_*(\mu_N^*(\T_S)) & \longrightarrow &  \varphi_*(\N_{\varphi})    &  \longrightarrow 0.    
\end{array}
\end{equation}
This implies $\T_S \langle X \rangle \cong {\mu_N}_*(\mu_N^*(\T_S) \langle C \rangle )$. Since $\mu_N$ is birational, 
by Leray's isomorphism we obtain \eqref{eq:iso}.

The last part of the statement directly follows from Proposition \ref{prop:beau}. 
\end{proof}

Let now  $V_{p,\delta} \to \mathcal{V}_{p,\delta}$ be an {\'e}tale atlas and let 
\begin{equation}\label{eq:diag}
\xymatrix{
\mathcal X  \ar[dr]_{^{\rho}} \ar@{^{(}->}[r]  & \mathcal S \ar[d]  \\
& V_{p,\delta},
}
\end{equation}
be the family induced by the universal one.  Since $V_{p,\delta}$ is a smooth, in particular normal,  scheme
we are in position to apply the results in \cite[p.\;80]{T}, i.e.
there exists a commutative diagram
\begin{equation}\label{eq:diag'}
\xymatrix{
\mathcal C  \ar[dr]_{\widetilde{\rho}} \ar[r]^{\Phi} & \mathcal X \ar[d]^{\rho} \\
& V_{p,\delta},
}
\end{equation}
where

\begin{itemize}
\item $\Phi$ is the normalization morphism,
\item $\widetilde{\rho}$ is smooth,
\item $\widetilde{\rho}$ gives
the simultaneous desingularization of the fibres of the family $\rho$; namely, for each
$v \in V_{p,\delta}$, ${\mathcal C}(v)$ is a smooth curve of genus $g = p- \delta$, which is
the normalization of the irreducible, nodal curve ${\mathcal X}(v)$.
\end{itemize}

Thus, as in \eqref{eq:cp}, the family $\widetilde{\rho}$ defines the natural morphism
\begin{equation}\label{eq:cg}
\xymatrix{ V_{p,\delta}
\ar@{->}[r]^{c_{p,\delta}} &  \M_g, }
\end{equation}by sending $v$ to the birational isomorphism class $[{\mathcal C}(v)] \in \M_g$.

As discussed in \S\,\ref{S:MukBea} for smooth curves, 
by Propositions \ref{prop:beau} and \ref{prop:flam}, and by passing to cohomology in the left-hand-side 
column of diagram \eqref{eq:diag1},
the differential of the morphism $c_{p,\delta}$ at a point $v\in V_{p,\delta}$ parametrizing  a pair $(S,X)$  can be identified with the cohomology map
\begin{equation}\label{eq:cg*}
H^1(\mu_N^*(\T_S) \langle C\rangle) \stackrel{H^1(\tau)}{\longrightarrow} H^1(\T_C).
\end{equation}

As in \eqref{eq:dimp}, we have
\begin{eqnarray}\label{eq:dimg}
\begin{array}{ccl}
 {\rm dim}( V_{p,\delta}) > {\rm dim} (\M_g), & \; {\rm for} \; & g \leq 10, \\
{\rm dim}( V_{p,\delta}) = {\rm dim} (\M_g), &  \; {\rm for} \; & g = 11, \\
{\rm dim}( V_{p,\delta}) < {\rm dim} (\M_g), & \; {\rm for} \; & g \geq 12.
\end{array}
\end{eqnarray}

In particular, independently from $p \geq 3$, for any $g \leq 11$ the morphism $c_{p,\delta}$ is expected to be dominant. Moreover,  as in the smooth case of \S\;\ref{S:MukBea}, if we denote by $F_{p,\delta}$ the general fibre of $c_{p,\delta}$ then, 
for any $ p \geq 3$, the {\em expected dimension} of $F_{p,\delta}$ is:

\begin{equation}\label{eq:expdim2}
{\rm expdim} (F_{p,\delta}) = \Big\{ \begin{array}{ccl}
22 - 2 g & \; {\rm for} \; & g \leq 10, \; p \geq  g +1 ,\\
 0 & \; {\rm for} \; & g \geq 11, \; p \geq g+1 .
\end{array}
\end{equation}

\section{General moduli for $g < p \leq 11$}\label{S:p<12} 

The aim of this section is to give some partial affirmative answers to the above expectations. Namely, we show that 
on a general, primitively polarized $K3$ surface of genus $ 3 \leq p \leq 11$, the normalizations of $\delta$-nodal 
curves in $|H|$, with $\delta >0$ as in \eqref{eq:boundpdelta}, define families of smooth curves with general moduli. 

Precisely, by recalling Proposition \ref{prop:smst}, we have:

\begin{theorem}\label{thm:main} Let $ 3 \leq p \leq 11$ be an integer. Let $\delta  $ and 
$g = p- \delta $ be positive integers as in \eqref{eq:boundpdelta}. 

Let $V \subseteq V_{p,\delta}$ be any irreducible component and let
$$c_{p,\delta}|_{V} : V \longrightarrow \M_g$$be the restriction to $V$ of the morphism
$c_{p,\delta}$ as in \eqref{eq:cg}.

Then, for any $ 2 \leq g < p  \leq  11$, $c_{p,\delta}|_V$ is dominant. In particular, for the general 
fibre $F_{p,\delta}$ of $c_{p,\delta}|_V$, we have:
$${\rm dim}(F_{p,\delta}) = 22-2g = {\rm expdim}(F_{p,\delta}).$$ 
\end{theorem}

\begin{remark}\label{rem:main} {\normalfont (1) Recall that, if $(S,H)$ is general in $\B_{10}$, smooth curves in 
$|H|$ are not with general
moduli (cf. Theorem \ref{thm:mukai}(ii)). On the contrary, from Theorem \ref{thm:main}, if $(S,H)$ is general in 
$\B_{11}$, then nodal curves in $V_{|H|,1}$ have normalizations of geometric genus $g=10$ that are curves with 
general moduli. 

\noindent
(2) At the same time, if $(S,H)$ is general in $\B_{10}$, from Theorem \ref{thm:main}, 
it follows that the general, irreducible, $\delta$-nodal curve in the linear system 
$|H|$ has a normalization of genus $g = 10 - \delta $ with general moduli. 

In particular, we have the following situation: consider the rational map
$$\overline{\KC}_{10} \dashrightarrow \overline{\M}_{10},$$which is defined on the open substack 
$\overline{\KC}_{10}^0$ of pairs $(S,X)$ s.t. $S = (S,H)$ is in $\B_{10}^0$ and $X$ is {\em nodal} and irreducible. 
Let$$ \overline{c}_{10}: \overline{\KC}_{10}^0 \longrightarrow \overline{\M}_{10}$$be the induced morphism. 
Then
${\rm Im} (\overline{c}_{10})$ is a divisor in $\overline{\M}_{10}$ such that 
$${\rm Im} (\overline{c}_{10}) \cap \partial \overline{\M}_{10} \subset \Delta_0$$ is a divisor whose general 
element has normalization a general curve of genus $9$. 

Since this divisor has dimension $${\rm dim} (\Delta_0) - 1 = {\rm dim} (\overline{\M}_{10}) - 2 = 25,$$the universal 
curve
\begin{displaymath}
\begin{array}{c}
{\mathcal X} \\ 
\downarrow \\ 
{\rm Im} (\overline{c}_{10}) \cap \Delta_0
\end{array}
\end{displaymath}induces a rational map 
$$({\rm Im} (\overline{c}_{10}) \cap \Delta_0) \dashrightarrow \M_9$$
that is dominant by Theorem \ref{thm:main}, and whose general fibre has dimension $1$. Precisely, 
this fibre determines a $1$-dimensional subscheme in the second symmetric product of the curve of genus $9$ 
parametrized by the image point in $\M_9$.

}

\end{remark}

\begin{proof}[Proof of Theorem \ref{thm:main}] From  \eqref{eq:cg*} and the left vertical column in \eqref{eq:diag1}, a sufficient condition for the surjectivity of the differential of $c_{p,\delta}|_V$ at a point 
$v \in V$ parametrizing a pair $(S,X)$ is that
\begin{equation}\label{eq:cgsu}
H^2(\mu_N^*(\T_S) (-C)) = (0).
\end{equation}If $E= \sum_{i=1}^{\delta} E_i$ denotes the $\mu_N$-exceptional divisor on $\widetilde{S}$, 
by Serre duality, $$H^2(\mu_N^*(\T_S) (-C)) \cong H^0(\mu_N^*(\Omega^1_S)(C + E))
\cong H^0(\mu_N^*(\Omega^1_S (H))(- E)),$$since $\mu_N^*(X) = C + 2 E$
and since $X \sim H$ on $S$. 

By the Leray isomorphism, $$ H^0(\mu_N^*(\Omega^1_S (H))(- E))
\cong H^0(\Ii_{N/S} \otimes \Omega_S^1(H)),$$where $\Ii_{N/S}$ denotes the ideal sheaf of $N $ in $S$.
One has the exact sequence:
\begin{equation}\label{eq:aiuto}
0 \longrightarrow  \Ii_{N/S} \otimes \Omega_S^1(H) \longrightarrow \Omega^1_S(H) \longrightarrow \Omega_S^1(H)|_N \longrightarrow 0. 
\end{equation}

For what concerns the cases either $ 3 \leq p \leq 9$ or $p =11$, from \eqref{eq:aiuto} and from 
Proposition \ref{prop:beaumuk}(i), we get that also $H^0(\Ii_{N/S} \otimes \Omega_S^1(H)) = (0)$, 
which concludes the proof in this case. 

For what concerns the case $p =10$, from Proposition \ref{prop:beaumuk}(ii) we know that, 
if $(S,H)$ is general in $\B_{10}$, then $H^0(\Omega^1_S(H)) \cong \CC$. For any $\delta$ as above, let 
$W \subseteq V_{|H|, \delta}$ be any irreducible component of the Severi variety of irreducible, $\delta$-nodal curves in $|H|$ on $S$. Let $[X] \in W$ be a general point and let $N = \Sing(X)$. 

From the exact sequence \eqref{eq:aiuto}, we know that 
$$ H^0(\Ii_{N/S} \otimes \Omega_S^1(H)) \hookrightarrow H^0(\Omega^1_S(H)) \cong \CC.$$We claim that this implies 
$h^0(\Ii_{N/S} \otimes \Omega_S^1(H)) = 0$, so the statement. 

Indeed, if $h^0(\Ii_{N/S} \otimes \Omega_S^1(H)) \neq  0$, then we would have 
$$H^0(\Ii_{N/S} \otimes \Omega_S^1(H)) \cong H^0(\Omega^1_S(H)),$$which means that 
the unique (up to scalar multiplication) non-zero global section of $H^0(\Omega^1_S(H))$ would pass through $N = \Sing(X)$. Since $X$ is general in a $g \geq 2$ dimensional family of curves, this is a contradiction. 
\end{proof}

Our result naturally leads to some questions.

\begin{question}\label{rem:ostruzione}
{\em When $X$ is smooth, the surjectivity of the Wahl map
$$
 \phi_{\omega_X} : \wedge^2 H^0(X,\omega_X)
 \longrightarrow H^0(X, \omega_X^{\otimes 3}) 
$$
is an obstruction to embed  $X$ in a $K3$ surface
(see \cite{W} and also \cite{BM}). 
It is natural to ask whether a similar result holds in the singular case. 
Precisely it would be interesting to understand whether there exists a Wahl-type
obstruction for a  smooth curve to have a nodal model lying on a $K3$ surface. 
}
\end{question}

\begin{question}\label{rem:iperell}
{\em Does the image of the map $c_{p,\delta}$ meet the hyperelliptic locus $\Ha_g$ in $\M_g$ ? If yes, and $(S_0,X_0)$ is a pair mapped into the 
hyperelliptic locus,  does the dimension 
of the locus of curves in 
$V_{|X_0|,\delta}(S_0)$ having hyperelliptic normalizations coincide with the expected one, which  is two, cf.
\cite[Lemma 5.1]{fkpS}? In fact, in  \cite[Lemma 5.1]{fkpS} we show that the dimension is two if $\Pic(S) \iso \ZZ[X_0]$. 
Answers to these questions, 
as explained in \cite[\S 6]{fkpS}, would yield a better understanding of rational curves  in the Hilbert square 
$S^{[2]}$ of the general $K3$ surface $S$, and then of its Mori cone ${\textrm{NE}}(S^{[2]})$.
The answers seem to be subtle, as they do not depend only on the geometric genus of $X_0$ but also on the number of nodes. For instance, for a general $(S,H)\in \B_p$, the locus of curves in 
$V_{|H|,\delta}(S)$ having hyperelliptic normalizations is {\it empty} when $\delta \leq (p-3)/2$ (\cite[Theorem 1]{fkp}),
and {\it nonempty}, and two-dimensional, for $p-3 \leq \delta \leq p-2$ (for $\delta =p-2$ is obvious; for $\delta = p-3$ cf. \cite[Theorem 5.2]{fkpS}). Of course, similar
questions may be asked for other gonality strata in $\M_g$.
}
\end{question}

\begin{question}\label{quest:pd} {\normalfont What about the cases $p \geq 12$? From \eqref{eq:dimg} and 
Theorem \ref{thm:main}, one could in principle expect the following situation: 

\vskip 10pt 

\noindent
{\em Let $p \geq 12$ be an integer. Let $\delta  $ and 
$g = p- \delta $ be positive integers as in \eqref{eq:boundpdelta}. For any irreducible 
component $V \subseteq V_{p,\delta}$, consider 
$$c_{p,\delta}|_V : V \longrightarrow \M_g.$$Then: 

\begin{itemize}
\item[(i)] for any $2 \leq g \leq 11 $, the 
morphism $c_{p,\delta}|_V$ is dominant.

\item[(ii)] for any $  12 \leq g <p $, $c_{p,\delta}|_V$ is generically finite.
\end{itemize} 
}

\vskip 7pt

Possible approaches to investigate the above expectations are the following: 

\vskip 5pt 

\noindent
$\bullet$ for (i), as in the proof of Theorem \ref{thm:main}, 
a sufficient condition for $c_{p,\delta}|_V$ to be smooth at a point $v \in V$ parametrizing a pair $(S,X)$ is 
$H^0(\Ii_{N/S} \otimes \Omega_S^1(H))  = (0)$, where $N = \Sing(X)$. By Remark \ref{rem:accazero}, one can 
say that 
\[ h^0((\Ii_{N/S} \otimes \Omega_S^1(H)) \geq h^0(\Omega_S^1(H))-2\delta \geq 2(p-\delta)-22=2g-22\](note that 
$2g - 22 \leq 0$ by assumption). On the other hand,  
\[ h^0((\Ii_{N/S} \otimes \Omega_S^1(H)) >0 \; \mbox{if} \;  g \geq 12, \]as it must be (cf. \eqref{eq:dimg}); 

\vskip 5pt

\noindent
$\bullet$ for (ii) above, a sufficient condition for $c_{p,\delta}|_V$ to be unramified at a point $v \in V$ parametrizing a pair $(S,X)$ is $H^1(\Ii_{N/S} \otimes \Omega_S^1(H))  = (0)$, where $N = \Sing(X)$.  
Observe that, for $p \geq 13$, from \eqref{eq:aiuto} we have the exact sequence
$$ \cdots \longrightarrow H^0(\Omega^1_S(H) ) \stackrel{ev^2_N}{\longrightarrow} H^0(\Omega_S^1(H)|_N) \longrightarrow H^1(\Ii_{N/S} 
\otimes \Omega_S^1(H)) \longrightarrow 0,$$since $H^1(\Omega_S^1(H)) = (0) $ from Proposition \ref{prop:beaumuk}(iii). 
Thus, the surjectivity of the evaluation morphism $ev^2_{N}$ as above would imply that the differential of 
$c_{p,\delta}|_V$ at $v$ is unramified.

} 
\end{question}

%
%

\vskip 30pt

\noindent
{\small Flaminio Flamini, Dipartimento di Matematica, Universit\`a degli Studi di Roma "Tor Vergata", Viale
della Ricerca Scientifica, 1 - 00133 Roma, Italy. e-mail {\tt flamini@mat.uniroma2.it.}}

\vskip 10pt

\noindent
{\small Andreas Leopold Knutsen, Dipartimento di Matematica, 
Universit\`a degli Studi Roma Tre, Largo San 
Leonardo Mu\-rial\-do 1 - 00146 Roma, Italy. 
e-mail {\tt knutsen@mat.uniroma3.it.}}

\vskip 10pt

\noindent
{\small Gianluca Pacienza, Institut de Recherche Math\'ematique Avanc\'ee,
Universit\'e L. Pasteur et CNRS, rue R. Descartes - 67084 Strasbourg Cedex, France. e-mail {\tt pacienza@math.u-strasbg.fr.}}

\vskip 10pt

\noindent
{\small Edoardo Sernesi, Dipartimento di Matematica, 
Universit\`a degli Studi Roma Tre, Largo San 
Leonardo Murialdo 1 - 00146 Roma, Italy. 
e-mail {\tt sernesi@mat.uniroma3.it.}}

\end{document}